\documentclass{article}%
\usepackage{amssymb, enumerate}
\usepackage{amsmath}
\usepackage{amsfonts}
\usepackage{graphicx}
\usepackage[pagebackref=true]{hyperref}
\usepackage{color}
\usepackage[normalem]{ulem}%
\usepackage{amssymb}
\providecommand{\U}[1]{\protect\rule{.1in}{.1in}}
\providecommand{\U}[1]{\protect\rule{.1in}{.1in}}
\numberwithin{equation}{section}
\newtheorem{theorem}{Theorem}[section]

\newtheorem{corollary}[theorem]{Corollary}

\newtheorem{example}[theorem]{Example}

\newtheorem{lemma}[theorem]{Lemma}

\newtheorem{conjecture}[theorem]{Conjecture}
\newtheorem{prob}[theorem]{Problem}
\newtheorem{definition}[theorem]{Definition}
\newenvironment{proof}[1][Proof]{\noindent\textbf{#1.} }{\ \rule{0.5em}{0.5em}}
\begin{document}

\author{Do Trong Hoang\\Faculty of Mathematics and Informatics \\Hanoi University of Science and Technology\\1 Dai Co Viet, Hai Ba Trung, Hanoi, Vietnam\\hoang.dotrong@hust.edu.vn
\and Vadim E. Levit\\Department of Mathematics\\Ariel University, Israel\\levitv@ariel.ac.il
\and Eugen Mandrescu\\Department of Computer Science\\Holon Institute of Technology, Israel\\eugen\_m@hit.ac.il
\and My Hanh Pham\\Faculty of Education, An Giang University\\Vietnam National University Ho Chi Minh City\\Dong Xuyen, Long Xuyen, An Giang, Vietnam\\pmhanh@agu.edu.vn}
\title{Log-concavity of the independence polynomials of $\mathbf{W}_{p}$ graphs}
\date{}
\maketitle

\begin{abstract}
Let $G$ be a graph of order $n$. For a positive integer $p$, $G$ is said to be
a $\mathbf{W}_{p}$ graph if $n\geq p$ and every $p$ pairwise disjoint
independent sets of $G$ are contained within $p$ pairwise disjoint maximum
independent sets. In this paper, we establish that every connected
$\mathbf{W}_{p}$ graph $G$ is $p$-quasi-regularizable if and only if
$n\geq(p+1)\cdot\alpha$, where $\alpha$ is the independence number of $G$ and
$p\neq2$. This finding ensures that the independence polynomial of a connected
$\mathbf{W}_{p}$ graph $G$ is log-concave whenever $(p+1)\cdot\alpha\leq n\leq
p\cdot\alpha+2\sqrt{p\cdot\alpha+p}$ and $\frac{\alpha^{2}}{4\left(
\alpha+1\right)  }\leq p$, or $p\cdot\alpha+2\sqrt{p\cdot\alpha+p}<n\leq
\frac{\allowbreak\left(  \alpha^{2}+1\right)  \cdot p+\left(  \alpha-1\right)
^{2}}{\alpha-1}$ and $\frac{\alpha\left(  \alpha-1\right)  }{\alpha+1}\leq p$.
Moreover, the clique corona graph $G\circ K_{p}$ serves as an example of the
$\mathbf{W}_{p}$ graph class. We further demonstrate that the independence
polynomial of $G\circ K_{p}$ is always log-concave for sufficiently large $p$.


\end{abstract}

\section{Introduction}

Throughout this paper $G$ is a finite, undirected, loopless graph without
multiple edges, with vertex set $V(G)$ of cardinality $\left\vert V\left(
G\right)  \right\vert =n\left(  G\right)  $, and edge set $E(G)$ of size
$\left\vert E\left(  G\right)  \right\vert =m\left(  G\right)  $. An edge
$e\in E(G)$ connecting vertices $x$ and $y$ will also be denoted as $xy$ or
$yx$. In this case, $x$ and $y$ are said to be \textit{adjacent}. An
\textit{independent set} in $G$ is a set of pairwise non-adjacent vertices. A
\textit{maximal independent set} of $G$ is the one that cannot be extended
further. A largest independent set in a graph is called a \textit{maximum
independent set}, and its cardinality is denoted by $\alpha(G)$, known as the
\textit{independence number} of $G$. It is well known that $\alpha(G) = 1 $ if
and only if $G$ is a complete graph.

Let $A$ be a subset of $V(G)$. We denote $G[A]$ the induced subgraph of $G$ on
$A$, while by $G-A$ we mean $G[V(G)-A]$. The neighborhood of $A$ in $G$ is the
set
\[
N_{G}(A)=\{v:v\in V(G)-A\text{ and }uv\in E(G)\text{ for some }u\in A\},
\]
the closed neighborhood of $S$ is $N_{G}[A]=A\cup N_{G}(A)$, and the
\textit{localization} of $G$ with respect to $A$ is $G_{A}=G-N_{G}[A]$. If
$A=\{v\}$, we write $N_{G}(v)$ (resp. $N_{G}[v]$, $G_{v}$, $G-v$) instead of
$N_{G}(\{v\})$ (resp. $N_{G}[\{v\}]$, $G_{\{v\}}$, $G-\{v\}$). The number
$\deg_{G}(v)=\left\vert N_{G}(v)\right\vert $ is called the \textit{degree} of
$v$ in $G$. A vertex of degree zero is an \textit{isolated vertex}. In
addition, $\delta(G)$ is the minimum degree of vertices of $G$.

A graph is called \textit{well-covered} if all its maximal independent sets
have the same size \cite{Plummer1970, Plum}. For instance, all complete graph
on $n$ vertices, denoted by $K_{n}$, is well-covered and the only cycles that
are well-covered are $C_{3},C_{4},C_{5},$ and $C_{7}$. Additionally, the path
on $n$ vertices, denoted by $P_{n}$, is well-covered if and only if $n=1,2,4$.

In this paper, we extend the exploration of well-covered graphs. In 1975,
Staples introduced the hereditary $\mathbf{W}_{p}$ as a generalization of
well-covered graphs \cite{StaplesThesis,Staples}. For a positive integer $p$,
a graph $G$ belongs to class $\mathbf{W}_{p}$ if $n(G)\geq p$ and for every
$p$ pairwise disjoint independent sets $A_{1},\ldots,A_{p}$ there exist $p$
pairwise disjoint maximum independent sets $S_{1},\ldots,S_{p}$ in $G$ such
that $A_{i}\subseteq S_{i}$ for $1\leq i\leq p$. The graph $G\in\mathbf{W}%
_{p}$ is also called the $\mathbf{W}_{p}$\textit{\ graph}. Observe, by
definition, that a graph is $\mathbf{W}_{1}$ if and only if it is
well-covered, and
\[
\mathbf{W}_{1}\supseteq\mathbf{W}_{2}\supseteq\mathbf{W}_{3}\supseteq
\cdots\supseteq\mathbf{W}_{p}\supseteq\cdots.
\]
Various methods for constructing $\mathbf{W}_{p}$ graphs are presented in
details in \cite{Pinter1991,Staples}.

A well-covered graph (with at least two vertices) is $1$-\textit{well-covered}
if the deletion of any vertex of the graph leaves it well-covered
\cite{Staples}. For instance, $K_{2}$ is $1$-well-covered, while a path on
four vertices $P_{4}$ is well-covered, but not $1$-well-covered. Furthermore,
the close relationship between $\mathbf{W}_{2}$ graphs and $1$-well-covered
graphs is demonstrated as follows.

\begin{theorem}
\cite{StaplesThesis,Staples}\label{th4} Let $G$ be a graph without isolated
vertices. Then $G$ is $1$-well-covered if and only if $\ G$ is a
$\mathbf{W}_{2}$ graph.
\end{theorem}

The $\mathbf{W}_{p}$ graphs exhibit numerous intriguing properties and
characteristics, both in combinatorics and commutative algebra, as detailed in
\cite{HT, W2LM, Pinter, Pinter2, Plum}. To illustrate, consider the polynomial
ring $R=K[x_{1},\ldots,x_{n}]$ in $n$ variables over a field $K$, and let $G$
be a graph with vertex set $V=\{1,\ldots,n\}$. We associate to the graph $G$ a
quadratic square-free monomial ideal
\[
I(G)=(x_{i}x_{j}\mid ij\in E(G))\subseteq R,
\]
which is called the \textit{edge ideal} of $G$. We say that $G$ is
\textit{Cohen--Macaulay} (resp. \textit{Gorenstein}) if $I(G)$ is a
Cohen--Macaulay (resp. Gorenstein) ideal. Notably, every Gorenstein graph is
Cohen--Macaulay, while the converse is not generally true.

It has been established that $G $ is well-covered whenever it is
Cohen--Macaulay \cite[Proposition 6.1.21]{Vi}, and $G $ belongs to
$\mathbf{W}_{2} $ whenever it is Gorenstein \cite[Lemma 2.5]{HT}. For
triangle-free graphs, $\mathbf{W}_{2} $ graphs are also Gorenstein
\cite[Theorem 4.4]{HT}. However, in general, not all $\mathbf{W}_{2} $ graphs
are Gorenstein. An example of a $\mathbf{W}_{p} $ graph that is always
Cohen--Macaulay is the clique corona graph. Let $\mathcal{H} = \{H_{v} : v \in
V(G)\}$ be a family of non-empty graphs indexed by the vertex set of a graph
$G$. The \textit{corona} $G \circ\mathcal{H}$ of $G$ and $\mathcal{H}$ is
defined as the disjoint union of $G$ and $H_{v}$ for each $v \in V(G)$, with
additional edges connecting each vertex $v \in V(G)$ to all the vertices of
$H_{v}$ \cite{FruchtHarary}. When all graphs $H_{v}$ in $\mathcal{H}$ are
complete graphs, $G \circ\mathcal{H} $ is referred to as a \textit{clique
corona graph}. A clique corona graph is not only a well-covered graph
\cite[Theorem 1]{TV92}, but also a Cohen--Macaulay graph \cite[Theorem
2.6]{DP}. If $H_{v} = K_{p}$ for every $v \in V(G)$, we use $G \circ K_{p}$ to
denote $G \circ\mathcal{H} $.

\begin{definition}
\cite{LM2017} For $\lambda>0$, a graph $G$ is $\lambda$-quasi-regularizable
if
\[
\lambda\cdot\left\vert S\right\vert \leq|N_{G}(S)|,
\]
for every independent set $S$ of $G$.
\end{definition}

If $\lambda=1$, then $G$ is said to be a \textit{quasi-regularizable} graph
\cite{Berge}.

\begin{theorem}
\cite{Berge}\label{Berge} Every well-covered graph without isolated vertices
is quasi-regularizable.
\end{theorem}

It is noteworthy pointing out that the structure of $\lambda$%
-quasi-regularizable graph implies the correlation between the number of
vertices with the independence number. In particular, by the theorem above,
one obtains that $n\left(  G\right)  \geq2\alpha\left(  G\right)  $ holds for
every well-covered graph $G$. Furthermore, if $G$ is a $\mathbf{W}_{p}$ graph
with $p\geq2$, then it is $(p-1)$-quasi-regularizable \cite[Theorem 2.6
(iii)]{DLMP}. We have also conjectured that $G$ is $p$-quasi-regularizable
\cite[Conjecture 2.7]{DLMP}. However, it turned out that this assertion does
not hold in general. For example, $C_{5}\in\mathbf{W}_{2}$, but it is not
$2$-quasi-regularizable. Nonetheless, in Section \ref{sec2}, we confirm the following.

\textbf{Theorem 2.8} \textit{Let }$G$\textit{\ be a connected }$W_{p}%
$\textit{\ graph, where }$p\neq2$\textit{. Then }$G$\textit{\ is }%
$p$\textit{-quasi-regularizable if and only if }$n(G)\geq(p+1)\cdot\alpha
(G)$\textit{.}

Beyond exploring graph structures, graph theory presents numerous intriguing
problems related to log-concavity, a concept that has deep implications in
combinatorics and algebra \cite{S2}. For instance, it is well-known that the
matching polynomial of a graph has only real zeros, making it log-concave
\cite{HL72}. This property is significant for understanding graph matchings
and their applications in network theory. Recently, a major breakthrough was
achieved with the resolution of the log-concavity of the chromatic polynomial
of a graph \cite{Huh12}, which has important implications for graph coloring
and phase transitions in statistical physics. Another key polynomial
associated with a graph is the \textit{independence polynomial} of a graph
$G$, denoted $I(G;x) $, which is defined in \cite{GU2} as follows:
\[
I(G;x)=\sum_{k=0}^{\alpha(G)} s_{k} x^{k} = s_{0} + s_{1}x + \cdots+
s_{\alpha(G)}x^{\alpha(G)},
\]
where $s_{k}$ represents the number of independent sets of cardinality $k$ in
the graph $G$. The independence polynomial $I(G;x)$ is said to be:

\begin{itemize}
\item \textit{log-concave} if $s_{k}^{2}\geq s_{k-1}\cdot s_{k+1}$ for all
$1\leq k\leq\alpha(G)-1$; or

\item \textit{unimodal} if there exists an index $0\leq k\leq\alpha(G)$ such
that
\[
s_{0}\leq\cdots\leq s_{k-1}\leq s_{k}\geq s_{k+1}\geq\cdots\geq s_{\alpha(G)}.
\]

\end{itemize}

A well-known result by Chudnovsky and Seymour in \cite{CS} states that all the
roots of $I(G;x) $ are real whenever $G $ is a claw-free graph, which also
implies the log-concavity of $I(G;x) $ for all claw-free graphs $G $. The
study of the independence polynomial is a rich area with extensive literature,
including works addressing the log-concavity problems (see \cite{AMSE, CS, HL,
LM07, LM17, LM2017, Zhu, Zhu1} and their references).

\begin{lemma}
\cite{KG} \label{log-uni} If $P(x)$ is log-concave and $Q(x)$ is unimodal,
then $P(x)\cdot Q(x)$ is unimodal, while the product of two log-concave
polynomials is log-concave.
\end{lemma}

In \cite{AMSE}, Alavi, Malde, Schwenk, and Erd\"os proved that for any
permutation $\pi$ of $\{1,2,\dots,\alpha(G)\}$, there is a graph $G$ such
that
\[
s_{\pi(1)} < s_{\pi(2)} < \cdots< s_{\pi(\alpha(G))}.
\]

This result highlights the varied behaviors that graph polynomials can
display. Additionally, they conjectured that the independence polynomial
$I(G;x)$ is unimodal for any tree or forest $G$. This conjecture remains
unresolved and continues to inspire ongoing research. Recently, it was
demonstrated that there are infinite families of trees whose independence
polynomials are not log-concave \cite{KadLev}. This finding challenges earlier
assumptions and suggests new directions for exploring the conditions under
which log-concavity and unimodality hold. It is also worth noting that the
independence polynomials of some well-covered graphs are not log-concave
\cite{LM2017}, \cite{MT03}.

Revisiting the open conjecture of the unimodality of independence polynomials
of trees stated in \cite{AMSE}, it is known that $I(G;x)$ is log-concave
whenever $G$ is a well-covered spider \cite{LM04}. Additionally, Radcliffe
verified that the independence polynomials of trees with up to $25$ vertices
are log-concave \cite{Radcliffe}. Zhu and Chen, in \cite{Zhu1}, applied
factorization methods to show the log-concavity of independence polynomials in
some special cases of trees. In contrast, in \cite{KadLev}, the authors
demonstrated that there exist exactly two trees of order $26$ whose
independence polynomials are not log-concave.

It is known that for every well-covered graph $G$ whose $\mathrm{girth}%
(G)\ge6$ and $G\ne\{C_{7}, K_{1}\}$ then $G$ is well-covered if and only if $G
= H\circ K_{1}$ for some graph $H$ \cite{FHN93}. The conjecture about the
unimodality of $I(G\circ K_{1};x)$ for all graphs $G $ was stated in
\cite[Conjecture 3.3]{LM03}, validated in the case $\alpha\left(  G\right)
\leq4$ in \cite[Conjecture 3.3]{LM03} and extended to $5\leq\alpha\left(
G\right)  \leq8$ in \cite{ChenWnag2010}. In addition, the unimodality of the
independence polynomial of clique corona graphs $G\circ K_{p}$ was
investigated for various classes of graphs such as: $G$ is a claw-free graph
\cite[Corollary 3.12]{LM17}; $G$ is a quasi-regularizable graph with
$\alpha(G)\leq4$; $G$ is an arbitrary graph and $p$ satisfying $(p+1)(p+2)\geq
n(G)+1$; and $G$ is quasi-regularizable graph with $\alpha(G)\leq p+1$
\cite{DLMP}. As an application, it was proven that $I(S_{n}\circ K_{p};x)$ is
unimodal for all $p\geq\sqrt{n+1}-2$, where $S_{n}$ is the complete bipartite
graph $K_{1,n}$ \cite{DLMP}. Furthermore, for any graph $H$, taking into
account that $H\circ K_{p}\in\mathbf{W}_{p}$ \cite[Corollary 2.3]{DLMP}, the
log-concavity of independent polynomials of $\mathbf{W}_{p}$ graphs can be
further employed in this specific ones. Consequently, it yields that $I(G\circ
K_{p};x)$ is log-concave whenever $p$ is large enough in correspondence to
$n(G).$

The paper is organized as follows. In Section \ref{sec2}, we study some
structural properties of $\mathbf{W}_{p}$ graphs. Section \ref{sec3} presents
findings related to log-concave properties of $I(G;x)$ for $\mathbf{W}_{p}$
graphs $G$. As an application, the remainder of Section \ref{sec3} is
dedicated to examining the log-concavity of $I(H\circ K_{p};x)$ for
sufficiently large $p$. Finally, in Conclusion, we suggest potential
directions for future research.

\section{Quasi-regularizability of $\mathbf{W}_{p}$ graphs}

\label{sec2} The following characterization of the localization of a
$\mathbf{W}_{p} $ graph serves as a valuable tool in establishing the proof of
our main theorem in this section. Recall that $\mathbf{W}_{1}$ denotes the
family of all well-covered graphs. Several important results are known, as
summarized below:

\begin{lemma}
\label{CP} Let $G$ be a well-covered graph. Then

\begin{enumerate}
\item \cite{CP, Plum} $G_{v}$ is well-covered and $\alpha(G_{v})=\alpha(G)-1$
for all $v\in V(G)$.

\item \cite[Lemma 1]{FHN93} If $S$ is an independent set of $G$, then $G_{S}$
is well-covered and $\alpha(G)=\alpha(G_{S})+|S|$.
\end{enumerate}
\end{lemma}

\begin{lemma}
Let $G$ be a $\mathbf{W}_{2}$ graph with $\alpha(G)>1$. Then
\label{Pinter_lem}

\begin{enumerate}
\item \cite[Theorem 2]{Pinter} $G_{v}$ is also a $\mathbf{W}_{2}$ graph for
all $v\in V(G)$.

\item \cite[Lemma 3.3]{HT} If $S$ is an independent set of $G$ and $\left\vert
S\right\vert <\alpha(G)$, then $G_{S}$ is also a $\mathbf{W}_{2}$ graph.
\end{enumerate}
\end{lemma}

Furthermore, Staples offered essential evaluations for the general case of the
$\mathbf{W}_{p}$ class.

\begin{lemma}
\cite{Staples} \label{lem_equiv} Let $p\geq2$. Then $G$ is a $\mathbf{W}_{p}$
graph if and only if $G-v$ is a $\mathbf{W}_{p-1}$ graph and $\alpha
(G)=\alpha(G-v)$ for all $v\in V(G)$.
\end{lemma}

\begin{lemma}
\cite{Staples} \label{lem_key} Let $p\geq2$ and let $G$ be a connected
$\mathbf{W}_{p}$ graph of order $n$ with the independence number $\alpha$.
Then the following assertions are true.

\begin{enumerate}
\item[(i)] $n\geq p\cdot\alpha$. In particular, $n=p\cdot\alpha$ if and only
if $G$ is a complete graph on $p$ vertices.

\item[(ii)] If $\alpha>1$, then $\delta(G)\geq p$.
\end{enumerate}
\end{lemma}

Clearly, Lemma \ref{lem_key} holds true for well-covered graphs as well.

\begin{lemma}
\label{lem-connected}If $G$ is a $\mathbf{W}_{p}$ graph, then every connected
component of $G$ contains at least $p$ vertices, and, consequently, $n(G)\geq
p\cdot c\left(  G\right)  $, where $c\left(  G\right)  $ is the number of
connected components of $G$.
\end{lemma}

\begin{proof}
By the definition of $\mathbf{W}_{p}$ graphs, $n(G)\geq p$. Let $v_{1}%
,v_{2},\ldots,v_{p}\in V\left(  G\right) $. Hence, there exists $p$ pairwise
disjoint maximum independent sets of $G$, say $S_{1},S_{2},\ldots,S_{p}$, such
that $v_{i}\in S_{i}$ for all $1\leq i\leq p$.

Let $H$ be a connected component of $G$. Due to the maximum of the independent
set $S_{i}$, we have $\left\vert S_{i} \cap V(H) \right\vert > 0$ for all $1
\leq i \leq p$. Therefore, $S_{1} \cap V(H), S_{2} \cap V(H), \ldots, S_{p}
\cap V(H)$ are $p$ pairwise disjoint non-empty maximum independent sets in
$H$. This implies that $n(H) \geq p$, and consequently, $n(G) \geq p \cdot
c(G)$.
\end{proof}

\begin{theorem}
\label{disconnected} A graph is $\mathbf{W}_{p}$ if and only if each of its
every connected component is also $\mathbf{W}_{p}$.
\end{theorem}

\begin{proof}
Let us start by proving that if every connected component of a graph $G$ is
$\mathbf{W}_{p}$, then $G$ itself is $\mathbf{W}_{p}$. Indeed, for each
connected component $H$ of $G$, we have $n(G) \geq n(H) \geq p$. Now, let
$A_{1},\ldots,A_{p}$ be $p$ pairwise disjoint independent sets in $G$. Then,
$A_{1}\cap V\left(  H\right)  ,\ldots,A_{p}\cap V\left(  H\right)  $ may be
enlarged to $p$ pairwise disjoint maximum independent sets $S_{1}^{H}%
,\ldots,S_{p}^{H}$ in $H$, because $H$ $\in$ $\mathbf{W}_{p}$. Consequently,%
\[
A_{1}\subseteq%
{\displaystyle\bigcup\limits_{H}}
S_{1}^{H}, A_{2}\subseteq%
{\displaystyle\bigcup\limits_{H}}
S_{2}^{H},\ldots,A_{p}\subseteq%
{\displaystyle\bigcup\limits_{H}}
S_{p}^{H},
\]
where
\[%
{\displaystyle\bigcup\limits_{H}}
S_{1}^{H},%
{\displaystyle\bigcup\limits_{H}}
S_{2}^{H},\ldots,%
{\displaystyle\bigcup\limits_{H}}
S_{p}^{H}%
\]
are $p$ pairwise disjoint maximum independent sets in $G$.

Conversely, let $H$ be an arbitrary connected component of $G$. Appying Lemma
\ref{lem-connected}, we have $n(H) \geq p$. Second, let $A_{1}, \ldots, A_{p}$
be $p$ pairwise disjoint independent sets in $H$. Since each $A_{j}$ is also
an independent set in $G$, there exist $p$ pairwise disjoint maximum
independent sets $S_{1}, \ldots, S_{p}$ in $G$ such that $A_{i} \subseteq
S_{i}$ for $1 \leq i \leq p$, given that $G \in\mathbf{W}_{p}$. Consequently,
$S_{1} \cap V(H), \ldots, S_{p} \cap V(H)$ are $p$ pairwise disjoint maximum
independent sets in $H$ with $A_{i} \subseteq S_{i} \cap V(H)$ for $1 \leq i
\leq p$. Therefore, $H$ belongs to $\mathbf{W}_{p}$ as well.
\end{proof}

We are now in a position to prove the main theorem of this section. First, let
us present a further essential localization property of $\mathbf{W}_{p}$ class.

\begin{lemma}
\label{key2} Let $G$ be a $\mathbf{W}_{p}$ graph. The following assertions are true:

\begin{enumerate}
\item[(i)] if $\alpha(G)>1$, then $G_{x}\in\mathbf{W}_{p}$ for every $x\in
V(G)$;

\item[(ii)] if all connected components of $G$ have an independence number
greater than $1$, then $\delta(G)\geq p$;

\item[(iii)] if $S$ is an independent set of $G$ such that $\left\vert
S\right\vert <\alpha(G)$, then $G_{S}\in\mathbf{W}_{p}$. In particular, if
$p>1$, then $G_{S}$ has no isolated vertices.
\end{enumerate}
\end{lemma}

\begin{proof}
\emph{(i)} By Lemmas \ref{CP} and \ref{Pinter_lem}, the assertion holds for
$p=1,2$. Now we prove by induction on $p$. Assume that $p>2$. Let $x\in V(G)$.
According to Lemma \ref{lem_equiv}, it is enough to show that $G_{x}%
-v\in\mathbf{W}_{p-1}$ and $\alpha(G_{x}-v)=\alpha(G_{x})$, for all $v\in
V(G_{x})$.

First, we observe that $G_{x}-v=(G-v)_{x}$, as shown below:
\[
G_{x}-v=G-N_{G}\left[  x\right]  -v=G-v-N_{G-v}\left[  x\right]  =(G-v)_{x}.
\]

By Lemma \ref{lem_equiv}, $G-v\in\mathbf{W}_{p-1}$ and $\alpha(G)=\alpha
(G-v)$. By the induction hypothesis, $(G-v)_{x}\in\mathbf{W}_{p-1}$. Note that
$G$ is well-covered and thus $\alpha(G_{x})=\alpha(G)-1$ in accordance with
Lemma \ref{CP}.

Moreover, since $p>2$, then $G-v\in\mathbf{W}_{p-1}\subseteq\mathbf{W}_{1}$
which means that $G-v$ is well-covered. Thus, $\alpha((G-v)_{x})=\alpha
(G-v)-1$. Hence,
\begin{align*}
\alpha(G_{x}-v) &  =\alpha((G-v)_{x})=\alpha(G-v)-1\\
&  =\alpha(G)-1=\alpha(G_{x}).
\end{align*}

\emph{(ii)} By Theorem \ref{disconnected}, all connected components of $G$
belong to $\mathbf{W}_{p}$. Consequently, the assertion follows directly from
Lemma \ref{lem_key}\emph{(ii)}.

\emph{(iii)} We prove by induction on $\left\vert S\right\vert $. If
$S=\emptyset$, then the assertion holds trivially. Now suppose $S\neq
\emptyset$. Choose $x\in S$ and define $S^{\prime}=S-\{x\}$. Then, we have
\[
G_{S}=G_{x}-N_{G_{x}}[S^{\prime}]=(G_{x})_{S^{\prime}}.
\]
By the assertion \emph{(i)}, $G_{x}\in\mathbf{W}_{p}$. Applying the induction
hypothesis, $(G_{x})_{S^{\prime}}\in\mathbf{W}_{p}$ and, hence, $G_{S}%
\in\mathbf{W}_{p}$.

If a connected component of $G_{S}$ is a complete graph, then its order is $p
$ at least. Otherwise, it has no isolated vertices in accordance with the
assertion \emph{(ii)}.
\end{proof}

\begin{theorem}
\label{mthm} Let $G$ be a connected $\mathbf{W}_{p}$ graph with $p\neq2$. Then
$G$ is $p$-quasi-regularizable if and only if $n(G)\geq(p+1)\cdot\alpha(G)$.
\end{theorem}

\begin{proof}
Suppose $G$ is $p$-quasi-regularizable. Hence, if $S$ is a maximum independent
set of $G$, then we infer that
\[
n(G)=\left\vert N_{G}(S)\right\vert +\left\vert S\right\vert \geq
p\cdot\left\vert S\right\vert +\left\vert S\right\vert =\left(  p+1\right)
\cdot\left\vert S\right\vert =\left(  p+1\right)  \cdot\alpha(G).
\]

Conversely, suppose that $n(G)\geq(p+1)\cdot\alpha(G)$. Actually, the case
$p=1$ is Theorem \ref{Berge}. Assume that $p\geq3$. Clearly, if $\alpha(G)=1$,
then $G$ is a complete graph of order $n(G)\geq p+1$, and so $\left\vert
N_{G}(x)\right\vert \geq p$\ for all $x\in V(G)$.

Suppose that $\alpha(G)>1$. Let $S$ be a non-empty independent set of $G$. If
$\left\vert S\right\vert =1$, by Lemma \ref{lem_key}\emph{(ii)}, $\left\vert
N_{G}(S)\right\vert \geq p\cdot|S|$, as required.

If $\left\vert S\right\vert =\alpha(G)$, then $V(G)=S\cup N_{G}(S)$ and so
$\left\vert N_{G}(S)\right\vert =n(G)-\left\vert S\right\vert $. Hence,
$p\cdot\left\vert S\right\vert \leq|N_{G}(S)|$, whenever $n(G)\geq
(p+1)\cdot\alpha(G)$, as required.

Now we concentrate on the situation when $1<\left\vert S\right\vert
<\alpha(G)$. Then $\alpha(G)\geq3$. Let $x\in S$ be chosen arbitrarily. By
Lemma \ref{key2}\emph{(i)}, it follows that $G_{x}\in\mathbf{W}_{p}$ and
$\alpha(G_{x})=\alpha(G)-1$. Hence, $S-x$ is a non-empty independent set of
$G_{x}$. Therefore, by Lemma \ref{key2}\emph{(iii)}, $G_{S-x}\in\mathbf{W}%
_{p}$. Moreover, by Lemma \ref{CP}\emph{(ii)}, $\alpha(G_{S-x})=\alpha
(G)-\left\vert S-x\right\vert =\alpha(G)-\left\vert S\right\vert +1>1$. Let
\[
X_{x}=N_{G}(x)-N_{G}(S-x).
\]


\vskip0.5em \noindent\textit{Claim 1.} $\left\vert X_{x}\right\vert >0$. \vskip0.5em

Assume, to the contrary, that $X_{x}=\emptyset$. This implies that
$N_{G}(x)\subseteq N_{G}(S-x)$, and, therefore, $N_{G}(S)=N_{G}(S-x)$.
Consequently,
\[
V(G_{S-x})=\left(  V(G)-N_{G}[S]\right)  \cup\left\{  x\right\}  ,
\]
which indicates that $x$ is an isolated vertex of $G_{S-x}$. This is a
contradiction due to Lemma \ref{key2}\emph{(ii)}. Therefore, the set of
private neighbours of $x$ is not empty, i.e., $\left\vert X_{x}\right\vert >0
$.

\vskip0.5em \noindent\textit{Claim 2.} If $\left\vert X_{x}\right\vert \leq
p-1 $, then $\left\vert X_{x}\right\vert =p-1$. \vskip0.5em

In this case, $x$ is a vertex of degree at most $p-1$ in $G_{S-x}$. However,
$G_{S-x}\in\mathbf{W}_{p}$, and so, by Lemma \ref{key2}\emph{(ii)}, the
induced subgraph on $\{x\}\cup X_{x}$ in $G_{S-x}$ is a complete graph $K_{p}
$. Hence $\deg_{G_{S-x}}(x)=\left\vert X_{x}\right\vert =p-1$, as required.

\vskip0.5em Now, we consider the case where there exists a vertex $x\in S$
such that $\left\vert X_{x}\right\vert \geq p$. We assert that the inequality
$p\cdot\left\vert S-x\right\vert \leq\left\vert N_{G}(S-x)\right\vert $
implies the inequality $p\cdot\left\vert S\right\vert \leq\left\vert
N_{G}(S)\right\vert $, because
\begin{align*}
p\cdot\left\vert S-x\right\vert \leq\left\vert N_{G}(S-x)\right\vert  &
\Leftrightarrow p\cdot(\left\vert S\right\vert -1)\leq\left\vert
N_{G}(S)\right\vert -\left\vert X_{x}\right\vert \\
& \Leftrightarrow p\cdot\left\vert S\right\vert +(\left\vert X_{x}\right\vert
-p)\leq\left\vert N_{G}(S)\right\vert .
\end{align*}

In other words, for every independent set $S$ of $G$ with $\left\vert
S\right\vert <\alpha(G)$, and any vertex $x\in S$ such that $\left\vert
X_{x}\right\vert \geq p$, the inequality $p\cdot\left\vert S-x\right\vert
\leq\left\vert N_{G}(S-x)\right\vert $ implies $p\cdot\left\vert S\right\vert
\leq\left\vert N_{G}(S)\right\vert $ as well.

Now, apply this procedure to all vertices $x$ in the set $S$ for which
$\left\vert X_{x}\right\vert \geq p$. If every vertex in $S$ satisfies
$\left\vert X_{x}\right\vert \geq p$, the problem reduces to proving that
$p=p\cdot\left\vert A\right\vert \leq|N_{G}(A)|$ for all single-vertex subsets
$A$ of $V(G)$. This holds by Lemma \ref{lem_key}\emph{(ii)}, since $G $ is
connected. Thus, by \textit{Claim 2}, we may assume that $\left\vert
X_{x}\right\vert =p-1$ for all $x\in S$.

\vskip0.5em \noindent\textit{Fact 1.} $ab\notin E(G)$ for all $a\in%
{\textstyle\bigcup\limits_{x\in S}}
X_{x}$ and $b\in G_{S}$. \vskip0.5em

Indeed, the graph $G_{S-x}$ has a connected component $H$ that contains
$\{x\}\cup X_{x}$. By Lemma \ref{key2}\emph{(iii)}, $G_{S-x}$ belongs to
$\mathbf{W}_{p}$, and then Theorem \ref{disconnected} guarantees that $H$ is
also in $\mathbf{W}_{p}$. If there is an edge $ab$ connecting $X_{x}$ to
$G_{S}$, then $\{x,b\}$ forms an independent set of $H$. Consequently,
$\alpha(H)>1$, and since $\deg_{H}(x)=p-1$, which contradicts Lemma
\ref{key2}\emph{(ii)}.

\vskip0.5em \noindent\textit{Fact 2.} $X_{x}\cap X_{y}=\emptyset$ for all
distinct elements $x,y\in S$. \vskip0.5em

Indeed, this fact follows directly from the definition. \vskip0.5em \noindent
Let $U=N_{G}(S)-%
{\textstyle\bigcup\limits_{x\in S}}
X_{x}$. First, we claim that $U\neq\emptyset$. Indeed, $V(G_{S})\neq\emptyset
$, since $\alpha(G_{S})=\alpha(G)-\left\vert S\right\vert >0$. By \emph{Fact
1}, we must have $U\neq\emptyset$, because $G$ is connected. If $\left\vert
U\right\vert \geq\left\vert S\right\vert $, by \emph{Fact 2},
\[
\left\vert N_{G}(S)\right\vert =\sum_{x\in S}\left\vert X_{x}\right\vert
+\left\vert U\right\vert =(p-1)\cdot\left\vert S\right\vert +\left\vert
U\right\vert \geq p\cdot\left\vert S\right\vert .
\]
Now, we consider the case where $\left\vert U\right\vert \leq\left\vert
S\right\vert -1$ (and $\left\vert S\right\vert \geq2$). We will show that this
scenario cannot occur. By definition of the set $X_{x}$, every vertex in $U$
has two neighbors in $S$ at least. Then there exist two distinct vertices
$x,y$ in $S$ that are adjacent to a vertex $u_{xy}$ in $U$, since
$U\neq\emptyset$. Let $U_{xy}=U-N_{G}(S-x-y)$.

\vskip0.5em \noindent\emph{Fact 3.} If $U_{xy}\neq\emptyset$, then $\left\vert
U_{xy}\right\vert \geq2$. \vskip0.5em

\noindent Let $u_{xy}$ be a vertex in $U_{xy}$. Then $u_{xy}$ is not adjacent
to any vertex in $S-x-y$ but is adjacent to at least one of $x$ or $y$. By the
definition of $X_{x}$, it follows that $u_{xy}$ must be adjacent to both $x$
and $y$.

\vskip0.5em \noindent\emph{Case 1. } $N_{G}(u_{xy})\cap V(G_{S})\neq\emptyset
$. \vskip0.5em

In this case, let $A$ be a maximum independent set in $G_{S}$ that contains at
least one vertex from $N_{G}(u_{xy})\cap V(G_{S})$. Such $A$ exists, because
$G_{S}$ belongs to $\mathbf{W}_{p}$, and, consequently, it is well-covered.
Then, $A\cup(S-x-y)$ is an independent set in $G$, which implies that
$G_{A\cup(S-x-y)}\in\mathbf{W}_{p}$.

If $\left\vert U_{xy}\right\vert =1$, say $U_{xy}=\{u_{xy}\}$, then the vertex
set of $G_{A\cup(S-x-y)}$ is $\{x,y\}\cup X_{x}\cup X_{y}$. Noting that
$\deg_{G_{A\cup(S-x-y)}}(x)=p-1$, it follows that $G_{A\cup(S-x-y)}$ is
disconnected, consisting of two connected components $G[\{x\}\cup X_{x}]$ and
$G[\{y\}\cup X_{y}]$. Thus, both $G[\{x\}\cup X_{x}]$ and $G[\{y\}\cup X_{y}]$
are complete graphs $K_{p}$ and no edges exist between $X_{x}$ and $X_{y}$.
Clearly, $u_{xy}$ is adjacent to all vertices in $X_{x}\cup X_{y}$. Indeed,
suppose, to the contrary, that $X_{x}-N_{G}[u_{xy}]\neq\emptyset$
(respectively, $X_{y}-N_{G}[u_{xy}]\neq\emptyset$). Moreover, every vertex in
$X_{x}-N_{G}[u_{xy}]$ is neither adjacent to any vertex in $X_{y}-N_{G}%
[u_{xy}]$ nor in $V(G_{S})$, in accordance with \textit{Fact 1}. Since
$G_{u_{xy}}\in\mathbf{W}_{p}$, each of its connected components must be of
order $p$, at least, which contradicts the inequality $n\left(  G[X_{x}%
-N_{G}[u_{xy}]]\right)  <p-1$. Therefore, we must have $X_{x}-N_{G}%
[u_{xy}]=X_{y}-N_{G}[u_{xy}]=\emptyset$, meaning that $u_{xy}$ is adjacent to
all vertices in $X_{x}\cup X_{y}$. Now, it follows that $G_{(S-x-y)\cup
\{u_{xy}\}}=G_{S}-N_{G}[u_{xy}] $. Consequently, we obtain the following
inequality
\[
\alpha(G_{(S-x-y)\cup\{u_{xy}\}})\leq\alpha(G_{S})\Leftrightarrow
\alpha(G)-(\left\vert S\right\vert -1)\leq\alpha(G)-\left\vert S\right\vert ,
\]
which is a contradiction. Therefore, we must have $\left\vert U_{xy}%
\right\vert \geq2$.

\vskip0.5em \noindent\emph{Case 2. } $N_{G}(u_{xy}) \cap V(G_{S}) = \emptyset
$. \vskip0.5em

If $\left\vert U_{xy}\right\vert =1$, i.e., $U_{xy}=\{u_{xy}\}$, then, by
\textit{Fact 1}, $H=G\left[  \{u_{xy},x,y\}\cup X_{x}\cup X_{y}\right]  $ is a
connected component of $G_{S-x-y}$. Thus, $\{x,y\}$ is a maximal independent
set of $H$. Since $H\in\mathbf{W}_{p}\subset\mathbf{W}_{1}$, it follows that
$\{x,y\}$ is a maximum independent set of $H$, and, consequently,
$\alpha(H)=2$. Hence, there exists a vertex $a\in X_{x}\cup X_{y}$ such that
$\{u_{xy},a\}$ is a maximum independent set of $H$. Without loss of
generality, assume that $a\in X_{x}$. Then the subgraph $H_{u_{xy}}%
=H-N_{H}[u_{xy}]$ is a complete graph containing $a$, and similarly,
$H_{a}=H-N_{H}[a]$ is a complete graph containing $u_{xy}$.

First,%
\[
V\left(  H_{a}\right)  =\{y,u_{xy}\}\cup\left(  X_{y}-N_{G}(a)\right)
\cup\left(  X_{x}-N_{G}(a)-\left\{  a\right\}  \right)  .
\]
Hence, $X_{x}-N_{G}(a)-\left\{  a\right\}  =\emptyset$, because $H_{a}$ is
complete, and $y$ is not adjacent to any private neighbor of $x$. Thus
$X_{x}=N_{G}\left[  a\right]  $. Therefore,%
\[
n(H_{a})=\left\vert \{y,u_{xy}\}\right\vert +\left\vert X_{y}-N_{G}%
(a)\right\vert =2+\left\vert X_{y}-N_{G}(a)\right\vert .
\]
In addition, $X_{y}\cap N_{G}(u_{xy})\supseteq X_{y}-N_{G}(a)$, because
$\{a,u_{xy}\}$ is a dominating set in $H$. Consequently, we have%
\[
2+\left\vert X_{y}\cap N_{G}(u_{xy})\right\vert \geq2+\left\vert X_{y}%
-N_{G}(a)\right\vert =n(H_{a})\geq p,
\]
since $H_{a}\in\mathbf{W}_{p}$. Finally, $\left\vert X_{y}\cap N_{G}%
(u_{xy})\right\vert \geq p-2$.

Since $\left\vert X_{y}\right\vert =p-1$, there are two options left only:
either $\left\vert X_{y}\cap N_{G}(u_{xy})\right\vert =p-1$ or $\left\vert
X_{y}\cap N_{G}(u_{xy})\right\vert =p-2$.

If $\left\vert X_{y}\cap N_{G}(u_{xy})\right\vert =p-1$, then $X_{y}%
-N_{G}(u_{xy})=\emptyset$, because $\left\vert X_{y}\right\vert =p-1$. Hence,
\[
n(H_{u_{xy}})=\left\vert X_{x}-N_{G}(u_{xy})\right\vert ,
\]
since%
\[
V\left(  H_{u_{xy}}\right)  =\left(  X_{x}-N_{G}(u_{xy})\right)  \cup\left(
X_{y}-N_{G}(u_{xy})\right)  \text{.}%
\]

Finally, $n(H_{u_{xy}})=\left\vert X_{x}-N_{G}(u_{xy})\right\vert
\leq\left\vert X_{x}\right\vert =p-1$, which contradicts the assumption that
$H_{u_{xy}}\in\mathbf{W}_{p}$.

Therefore, we must have $\left\vert X_{y}\cap N_{G}(u_{xy})\right\vert =p-2$,
implying that $u_{xy}$ is adjacent to all vertices in $X_{y}-b$ for some $b\in
X_{y}$, i.e., $X_{y}-N_{G}(u_{xy})=\left\{  b\right\}  $. Hence, $N_{G}%
(u_{xy})\cap X_{x}=\emptyset$, because $\left\vert X_{x}\right\vert =p-1$ and
\[
\left\vert X_{x}-N_{G}(u_{xy})\right\vert +\left\vert \left\{  b\right\}
\right\vert =n\left(  H_{u_{xy}}\right)  \geq p,
\]
in order to give $H$ a chance to be a $\mathbf{W}_{p}$ graph. Thus, $V\left(
H_{u_{xy}}\right)  =X_{x}\cup\left\{  b\right\}  $. Therefore, $X_{x}\subseteq
N_{G}\left(  b\right)  $, since $H_{u_{xy}}$ is a complete graph.
Consequently, $H_{b}$ is a complete graph with%
\[
V\left(  H_{b}\right)  =\{u_{xy},x\}\cup\left(  X_{y}-N_{G}\left[  b\right]
\right)  ,
\]
which is possible only if $X_{y}-N_{G}\left[  b\right]  =\emptyset$, because
and $x$ is not adjacent to any private neighbor of $y$. Finally,
$H_{b}=\{u_{xy},x\}=K_{2}$, which belongs to $\mathbf{W}_{p}$, implying that
either $p=2$ or $p=1$. According to the assumption that $p\geq3$, this cannot
happen. \vskip0.5em By \emph{Fact 3}, what is left is to consider the case
$\left\vert U_{xy}\right\vert =0$, which implies that $U\subseteq
N_{G}(S-x-y)$.

Recall that $U\neq\emptyset$, because $G$ is connected. Further, since
$U\subseteq N_{G}(S-x-y)$, there must be a vertex $z\in S-x-y$ that is
adjacent to some $u\in U$. Now, let
\[
U_{xyz}=(N_{G}(x)\cap N_{G}(y)\cap N_{G}(z))-N_{G}(S-x-y-z).
\]

\vskip0.5em \noindent\emph{Fact 4.} Either $\left\vert U_{xz}\right\vert
\geq2$, or $\left\vert U_{yz}\right\vert \geq2$, or $\left\vert U_{xyz}%
\right\vert \geq3$. \vskip0.5em

We have
\[
V(G_{S-x-y})=\{x,y\}\cup X_{x}\cup X_{y}\cup V(G_{S}).
\]
By Lemma \ref{key2}\emph{(iii)}, $G_{S-x-y}\in\mathbf{W}_{p}$. Hence, the
graph $G_{S-x-y}$ consists of the disjoint union of complete subgraphs on the
vertex sets $\{x\}\cup X_{x}$, $\{y\}\cup X_{y}$, and the subgraph $G_{S} $.
This implies that both $X_{x}$ and $X_{y}$ are cliques, with no edges
connecting them.

If $u \in(N_{G}(z) \cap U) - N_{G}(S - x - y - z) $ and $u \notin U_{xz} \cup
U_{yz} $, by the definitions of $U_{xz} $ and $U_{yz} $, $u $ must belong to
both $N_{G}(S - x - z) $ and $N_{G}(S - y - z) $. Moreover $u $ is adjacent to
both $x $ and $y $, which implies that $u \in N_{G}(x) \cap N_{G}(y) \cap
N_{G}(z) $. Hence, we conclude that $u\in U_{xyz}$. Furthermore, by the
definitions of $X_{x}, X_{y}, $ and $X_{z} $, it follows that
\[
U_{xz} \cup U_{xyz} \cup U_{yz} = (N_{G}(z) \cap U) - N_{G}(S - x - y - z),
\]
and both $x$ and $z$ are adjacent to all vertices in $U_{xz}$, and both $y$
and $z$ are adjacent to all vertices in $U_{yz}$. Hence, we infer that
\[
V(G_{S-x-y-z})=\{x,y,z\}\cup X_{x}\cup X_{y}\cup X_{z}\cup U_{xz}\cup
U_{xyz}\cup U_{yz}\cup V(G_{S}).
\]
If $U_{yz}\neq\emptyset$ (resp. $U_{xz}\neq\emptyset$), by \emph{Fact 3}, we
know that $\left\vert U_{yz}\right\vert \geq2$ (resp. $\left\vert
U_{xz}\right\vert \geq2$), as expected. Conversely, suppose $U_{xz}%
=U_{yz}=\emptyset$. Thus, $U_{xyz}\neq\emptyset$. Let $H=G_{S-x-y-z} $. Then,
$H\in\mathbf{W}_{p}$ and its vertex set is
\[
V(H)=\{x,y,z\}\cup X_{x}\cup X_{y}\cup X_{z}\cup U_{xyz}\cup V(G_{S}).
\]
Since each of $H_{x}$, $H_{y}$, and $H_{z}$ consists of the disjoint union of
two complete graphs with vertex sets $\{x\}\cup X_{x}$, $\{y\}\cup X_{y}$, and
$\{z\}\cup X_{z}$, respectively, along with the graph $G_{S}$, it follows that
$X_{x},X_{y},X_{z}$ are cliques, and no edges exist between these sets.

If $\left\vert U_{xyz}\right\vert =1$, let $U_{xyz}=\{u\}$. By Lemma
\ref{key2}\emph{(ii)}, $u$ is adjacent to every vertex in $X_{x}\cup X_{y}\cup
X_{z}$. If $u$ is not adjacent to any vertex in $G_{S}$, then the connected
component of $H$ containing $\{x,y,z,u\}\cup X_{x}\cup X_{y}\cup X_{z}$ does
not belong to $\mathbf{W}_{p}$. Conversely, assume $u$ is adjacent to a vertex
in $V(G_{S})$. Then $H_{u}=G_{S}-N_{G}(u)$. This implies that
\begin{align*}
\alpha(H_{u})\leq\alpha(G_{S}) & \Leftrightarrow \alpha(H)-1\leq
\alpha(G)-|S|\\
& \Leftrightarrow \alpha(G)-\left\vert S-x-y-z\right\vert -1\leq
\alpha(G)-\left\vert S\right\vert \\
& \Leftrightarrow \alpha(G)-(\left\vert S\right\vert -3)-1\leq\alpha(G)-|S|,
\end{align*}
a contradiction.

If $\left\vert U_{xyz}\right\vert =2$, let $U_{xyz}=\{u,v\}$. Recall that
$H=G_{S-x-y-z}$ and
\[
V(H)=\{u,v,x,y,z\}\cup X_{x}\cup X_{y}\cup X_{z}\cup V(G_{S}).
\]
Examining the structure, we observe that $X_{x},X_{y},$ and $X_{z}$ are
cliques and there are no edges between these sets. If neither $u$ nor $v$ is
adjacent to any vertex in $G_{S}$, then $H$ consists of the disjoint union of
the connected component $K$ with vertex set $\{u,v,x,y,z\}\cup X_{x}\cup
X_{y}\cup X_{z}$ and the graph $G_{S}$. Note that $\alpha(K)=3$ and the sizes
of $(X_{x}-N_{K}(u))$, $(X_{y}-N_{K}(u))$, and $(X_{z}-N_{K}(u))$ are at most
$p-1$. Consequently, $K_{u}$ does not belong to $\mathbf{W}_{p}$. This implies
that at least one of $u$ or $v$ must be adjacent to a vertex in $G_{S}$.
Suppose that $u$ is adjacent to a vertex in $G_{S}$. Let $A$ be a maximum
independent set of $G_{S}$ that contains exactly one vertex from $N_{G}(u)\cap
V(G_{S})$. We consider the following two cases. \vskip0.5em \noindent
\emph{Case 1.} $v$ is not adjacent to a vertex in $A$. \vskip0.5em In this
case, $H_{A}$ is a graph in $\mathbf{W}_{p}$ induced by the vertex set
$\{v,x,y,z\}\cup X_{x}\cup X_{y}\cup X_{z}$, where $\{x,y,z\}$ forms a
domination set of $H_{A}$. Consequently, $H_{A\cup\{v\}}$ belongs to
$\mathbf{W}_{p}$ and is induced by the set $(X_{x}-N_{H}(v))\cup(X_{y}%
-N_{H}(v))\cup(X_{z}-N_{H}(v))$, with $\alpha(H_{A\cup\{v\}})=2$, which is a contradiction.

\vskip0.5em \noindent\emph{Case 2.} $v$ is adjacent to any vertex in $A$.
\vskip0.5em \noindent\emph{Case 2.1.} $uv \notin E(G) $. In this case, we
have
\[
V(H_{u}) = \{v\} \cup(X_{x} - N_{G}(u)) \cup(X_{y} - N_{G}(u)) \cup(X_{z} -
N_{G}(u)) \cup V(G_{S} - N_{G}(u)).
\]
Because $X_{x} - N_{G}(u)$, $X_{y} - N_{G}(u)$ and $X_{z} - N_{G}(u)$ are
cliques of size at most $p-1$ with no edges between them, and because $H_{u}$
is a $\mathbf{W}_{p}$ graph, we may assume without loss of generality that
$X_{x} - N_{G}(u) = X_{y} - N_{G}(u) = \emptyset$, leaving only $X_{z} -
N_{G}(u) = X_{z} $. This means that $v $ is connected to every vertex in
$X_{z} $ within $H_{u} $. Consequently, $H_{\{u,v\}} $ is an induced subgraph
of $G_{S} $ on the vertex set $G_{S}-N_{G}(u) - N_{G}(v)$, which leads to the
bound $\alpha(H_{\{u,v\}}) \leq\alpha(G). $ Thus, we obtain the inequality
$\alpha(G) - |S - x - y - z| - 2 \leq\alpha(G) - |S|, $ which leads to a contradiction.

\vskip0.5em \noindent\emph{Case 2.2.} $uv \in E(G) $. In this case, we have
\[
V(H_{u}) = (X_{x}-N_{G}(u)) \cup(X_{y}-N_{G}(u)) \cup(X_{z}-N_{G}(u)) \cup
V(G_{S}-N_{G}(u)).
\]
Therefore, we have $X_{x} = X_{y} = X_{z} = N_{G}(u) $, meaning that $u $ is
adjacent to every vertex in $X_{x} \cup X_{y} \cup X_{z} $. Consequently,
$H_{u}=G_{S} - N_{G}(u) $ is an induced subgraph of $G_{S} $. This gives the
inequality $\alpha(H_{u}) \leq\alpha(G_{S}),$ which simplifies to $\alpha(G) -
|S - x - y - z| - 1 \leq\alpha(G) - |S|, $ leading to a contradiction.

\vskip0.5em \noindent\emph{Fact 5.} For any distinct elements $x,y,x^{\prime
},y^{\prime},z^{\prime}$ in $S$, we have $U_{xy}\cap U_{x^{\prime}y^{\prime}%
}=\emptyset$ and $U_{xy}\cap U_{x^{\prime}y^{\prime}z^{\prime}}=\emptyset$.
\vskip0.5em This fact follows directly from the definition of $U_{xy}$ and
$U_{xyz}.$ \vskip0.5em \noindent By \emph{Fact 3} and \emph{Fact 4}, for any
two vertices $x,y$ in $S$, there exist at least two vertices in $U_{xy}$, or
for three vertices $x,y,z$ in $S$, there exist at least three vertices in
$U_{xyz}$. Now, consider another pair of vertices $x^{\prime},y^{\prime}$ in
$S-x-y$ or another trio $x^{\prime},y^{\prime},z^{\prime}$ in $S-x-y-z$. These
choices generate at least two vertices in $U_{x^{\prime}y^{\prime}}$ or at
least three vertices in $U_{x^{\prime}y^{\prime}z^{\prime}}$, respectively. By
\emph{Fact 5}, repeating this process iteratively, we obtain $\left\vert
S\right\vert \geq\left\vert U\right\vert $, which is a contradiction.
Therefore, this situation cannot arise.
\end{proof}

The \textquotedblleft connected\textquotedblright\ assumption in Theorem
\ref{mthm} is essential. For instance, the graph $G=K_{p}\cup K_{p+2}$ is a
$\mathbf{W}_{p}$ graph with $n(G)=2p+2$ and $\alpha(G)=2$, yet $G$ is not
$p$-quasi-regularizable. Moreover, in the context of Theorem \ref{mthm}, it is
worth noting that there exist $p$-quasi-regularizable graphs which do not
belong to $\mathbf{W}_{p}$. For instance, consider the graph $G$ depicted in
Figure \ref{fig123}. Clearly, $G$ is $2$-quasi-regularizable. However,
$G\notin\mathbf{W}_{2}$ because the disjoint independent sets $\{u\}$ and
$\{v\}$ cannot be extended to two disjoint maximum independent sets in $G$.

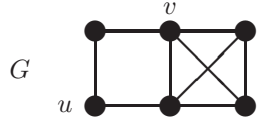
\begin{figure}[h]
\setlength{\unitlength}{1cm}\begin{picture}(5,1.2)\thicklines
\multiput(5,0)(1,0){3}{\circle*{0.29}}
\multiput(5,1)(1,0){3}{\circle*{0.29}}
\put(5,0){\line(1,0){2}}
\put(5,0){\line(0,1){1}}
\put(5,1){\line(1,0){2}}
\put(6,0){\line(0,1){1}}
\put(6,1){\line(1,-1){1}}
\put(6,0){\line(1,1){1}}
\put(7,0){\line(0,1){1}}
\put(4.6,0){\makebox(0,0){$u$}}
\put(6,1.3){\makebox(0,0){$v$}}
\put(4,0.5){\makebox(0,0){$G$}}
\end{picture}
\caption{$G$ is a $2$-quasi-regularizable graph.}%
\label{fig123}%
\end{figure}

\begin{corollary}
Let $H_{1},\ldots, H_{s}$ be connected components of a $\mathbf{W}_{p}$ graph
$G$ with $p\ne2$. Then $G$ is $p$-quasi-regularizable if and only if
$n(H_{i})\geq(p+1)\cdot\alpha(H_{i})$ for all $1\le i\le s$.
\end{corollary}

\begin{proof}
Let $S$ be an independent set of $G$. For each $1\leq i\leq s$, define
$S_{i}=S\cap V(H_{i})$. Since $S$ is independent in $G$, each $S_{i}$ is an
independent set in $H_{i}$. By Theorem \ref{mthm}, we have $\left\vert
N_{H_{i}}(S_{i})\right\vert \geq p\cdot|S_{i}|$. Summing over all $i$, it
follows that
\[
\left\vert N_{G}(S)\right\vert =\sum_{i=1}^{s}\left\vert N_{H_{i}}%
(S_{i})\right\vert \geq\sum_{i=1}^{s}p\cdot\left\vert S_{i}\right\vert
=p\cdot\sum_{i=1}^{s}\left\vert S_{i}\right\vert =p\cdot|S|.
\]
Conversely, for each $1\leq i\leq s$, let $S_{i}$ be an independent set of
$H_{i}$. Since $N_{H_{i}}(S_{i})=N_{G}(S_{i})$, the assumption implies that
$\left\vert N_{H_{i}}(S_{i})\right\vert \geq p\cdot|S_{i}|$. Applying Theorem
\ref{mthm} again, we conclude that $n(H_{i})\geq(p+1)\cdot\alpha(H_{i}).$
\end{proof}

\section{Log-concavity of independence polynomials of $\mathbf{W}_{p}$ graphs}

\label{sec3} In this section, the log-concavity problem is deeply concerned.
Let us recall some significant inequalities involving the coefficients of the
independent polynomial of a graph, for a later use.

\begin{lemma}
\cite[Lemma 1]{LM18} \label{lem2a} If $I(G;x)=\sum\limits_{k=0}^{\alpha\left(
G\right)  }s_{k}x^{k}$, then $\alpha\left(  G\right)  \cdot s_{\alpha\left(
G\right)  }\leq n\cdot s_{\alpha\left(  G\right)  -1}$.
\end{lemma}

The following lemma serves as an efficient tool that enables us to prove the
main theorem.

\begin{lemma}
\label{lemW2-reg} Let $I(G;x)=\sum\limits_{k=0}^{\alpha\left(  G\right)
}s_{k}x^{k}$. Then the following assertions are true:

\begin{enumerate}
\item[(i)] \cite[Theorem 2.1]{LM2017} if $G$ is $\lambda$-quasi-regularizable,
then
\[
(k+1)\cdot s_{k+1}\leq(n\left(  G\right)  -(\lambda+1)\cdot k)\cdot s_{k}%
\]
for all $0\leq k\leq\alpha\left(  G\right)  -1$;

\item[(ii)] \cite[Theorem 2.6]{DLMP} if $G$ is a connected graph in
$\mathbf{W}_{p}$, then
\[
p\cdot(\alpha\left(  G\right)  -k)\cdot s_{k}\leq(k+1)\cdot s_{k+1}%
\]
for all $1\leq k\leq\alpha\left(  G\right)  -1$.
\end{enumerate}
\end{lemma}

\begin{theorem}
\label{log2} Let $G$ be a connected $\mathbf{W}_{p}$ graph with $p\neq2$,
$n=n(G)$ and $\alpha=\alpha(G)$. If%
\[
\frac{\alpha^{2}}{4\left(  \alpha+1\right)  }\leq p\text{ and }n\in\left[
(p+1)\cdot\alpha,p\cdot\alpha+2\sqrt{p\cdot\alpha+p}\right]  \text{ \ \ \ }%
\]
or
\[
\frac{\alpha\left(  \alpha-1\right)  }{\alpha+1}\leq p\text{ and }n\in\left(
p\cdot\alpha+2\sqrt{p\cdot\alpha+p},\frac{\allowbreak\left(  \alpha
^{2}+1\right)  \cdot p+\left(  \alpha-1\right)  ^{2}}{\alpha-1}\right]  ,
\]
then $I(G;x)$ is log-concave.
\end{theorem}

\begin{proof}
If $\alpha=1$, then $G$ is a complete graph, and its independence polynomial
is given by $1+nx$. Consequently, this polynomial is log-concave. Therefore,
we may assume that $\alpha\geq2$. By Theorem \ref{mthm}, $G$ is $p$%
-quasi-regularizable, because $n\geq(p+1)\cdot\alpha$ and $G\in\mathbf{W}_{p}%
$. For all $1\leq k\leq\alpha-1$, Lemma \ref{lemW2-reg} implies
\begin{gather*}
(k+1)\cdot s_{k+1}\leq(n-(p+1)\cdot k)\cdot s_{k},\text{ and }\\
p\cdot(\alpha-k+1)\cdot s_{k-1}\leq k\cdot s_{k}.
\end{gather*}
Since $n\geq(p+1)\cdot\alpha$, then
\[
n-(p+1)\cdot k\geq n-(p+1)\cdot(\alpha-1)=n-(p+1)\cdot\alpha+p+1>0.
\]
Hence,
\[
s_{k}^{2}\geq\frac{(k+1)\cdot p\cdot(\alpha-k+1)}{k\cdot(n-(p+1)\cdot k)}\cdot
s_{k-1}\cdot s_{k+1}.
\]
We know that
\begin{align*}
\frac{(k+1)\cdot p\cdot(\alpha-k+1)}{k\cdot(n-(p+1)\cdot k)}\geq1  &
\Leftrightarrow(k+1)\cdot p\cdot(\alpha-k+1)\geq k\cdot(n-(p+1)\cdot k)\\
& \Leftrightarrow k^{2}-(n-p\cdot\alpha)\cdot k+p\cdot\alpha+p\geq0.
\end{align*}
Now, we consider the function
\[
f(k)=k^{2}-(n-p\cdot\alpha)\cdot k+p\cdot\alpha+p.
\]
Hence, its discriminant is
\[
\Delta=(n-p\cdot\alpha)^{2}-4\left(  p\cdot\alpha+p\right)  .
\]
\vskip0.5em \noindent\textit{Case 1}. $\Delta\leq0$. This is equivalent to
$n\leq p\cdot\alpha+2\sqrt{p\cdot\alpha+p}$. Consequently, $f(k)\geq0 $ for
every $k\in%
\mathbb{R}
$. Thus $s_{k}^{2}\geq s_{k-1}\cdot s_{k+1}$ for each $k\in\left\{
1,\ldots,\alpha-1\right\}  $, whenever%
\[
\left(  p+1\right)  \cdot\alpha\leq n\leq p\cdot\alpha+2\sqrt{p\cdot\alpha+p}.
\]

Clearly, the set of integers $n$ satisfying the above constraints is not empty
if and only if $\frac{\alpha^{2}}{4\left(  \alpha+1\right)  }\leq p$.
\vskip0.5em \noindent\textit{Case 2}. $\Delta>0$. This is equivalent to
$n>p\cdot\alpha+2\sqrt{p\cdot\alpha+p}$. Then, $f(k)\geq0$ for every
\[
k\leq k_{1}=\frac{n-p\cdot\alpha-\sqrt{\Delta}}{2}\text{ or }k\geq k_{2}%
=\frac{n-p\cdot\alpha+\sqrt{\Delta}}{2}.
\]

\vskip0.5em \noindent\textit{Subcase 2.1}. $\alpha-1\leq k_{1}$. \vskip0.5em
It means%
\begin{align*}
& \alpha-1\leq\frac{n-p\cdot\alpha-\sqrt{\Delta}}{2}\\
& \Leftrightarrow\sqrt{(n-p\cdot\alpha)^{2}-4\left(  p\cdot\alpha+p\right)
}\leq n-p\cdot\alpha-2\alpha+2,
\end{align*}
which may be true if and only if $p\cdot\alpha+2\alpha-2<n$, since $\Delta>0$.
If this constraint is satisfied, then we may continue as follows:
\begin{align*}
& (n-p\cdot\alpha)^{2}-4\left(  p\cdot\alpha+p\right)  \leq\left(
n-p\cdot\alpha-2\alpha+2\right)  ^{2}\\
& \Leftrightarrow0\leq4n+4p-8\alpha-4n\alpha+4\alpha^{2}+4p\alpha^{2}+4\\
& \Leftrightarrow n\leq\frac{\allowbreak\left(  \alpha^{2}+1\right)  \cdot
p+\left(  \alpha-1\right)  ^{2}}{\alpha-1}.
\end{align*}
Thus $s_{k}^{2}\geq s_{k-1}\cdot s_{k+1}$ for each $k\in\left\{
1,\ldots,\alpha-1\right\}  $, whenever%
\[
\Delta>0\text{ and }\alpha\cdot\left(  p+1\right)  +\left(  \alpha-1\right)
\leq n\leq\frac{\allowbreak\left(  \alpha^{2}+1\right)  \cdot p+\left(
\alpha-1\right)  ^{2}}{\alpha-1}.
\]
By the second constraint, if $\frac{\alpha\left(  \alpha-1\right)  }{\alpha
+1}>p$, then the set of integers $n$ satisfying the above constraints is
empty. On the other hand, if $\frac{\alpha\left(  \alpha-1\right)  }{\alpha
+1}\leq p$, then
\begin{align*}
& \alpha\cdot\left(  p+1\right)  +\left(  \alpha-2\right)  \leq p\cdot
\alpha+2\sqrt{p\cdot\alpha+p}\\
& \Leftrightarrow\alpha-1\leq\sqrt{p\cdot\alpha+p}\Leftrightarrow\alpha
^{2}-2\alpha+1\leq p\cdot\alpha+p
\end{align*}
which is true, because%
\[
\frac{\left(  \alpha-1\right)  ^{2}}{\alpha+1}\leq\frac{\alpha\left(
\alpha-1\right)  }{\alpha+1}\leq p.
\]
Thus $s_{k}^{2}\geq s_{k-1}\cdot s_{k+1}$ for each $k\in\left\{
1,\ldots,\alpha-1\right\}  $, whenever%
\[
p\cdot\alpha+2\sqrt{p\cdot\alpha+p}<n\leq\frac{\allowbreak\left(  \alpha
^{2}+1\right)  \cdot p+\left(  \alpha-1\right)  ^{2}}{\alpha-1}\text{ and
}\frac{\alpha\left(  \alpha-1\right)  }{\alpha+1}\leq p.
\]
\vskip0.5em \noindent\textit{Subcase 2.2}. $k_{2}\leq1$. \vskip0.5em It means%
\begin{align*}
& \frac{n-p\cdot\alpha+\sqrt{(n-p\cdot\alpha)^{2}-4\left(  p\cdot
\alpha+p\right)  }}{2}\leq1\\
& \Leftrightarrow\sqrt{(n-p\cdot\alpha)^{2}-4\left(  p\cdot\alpha+p\right)
}\leq2-n+p\alpha.
\end{align*}

It may be true if and only if $n<2+p\alpha$, because $\Delta>0$. On the other
hand, the inequality $\left(  p+1\right)  \cdot\alpha\leq n$ implies
$\alpha=1$, which contradicts our assumption that $\alpha\geq2$.

To conclude, it is worth mentioning that the inequality
\[
p\cdot\alpha+2\sqrt{p\cdot\alpha+p}\leq\frac{\left(  \alpha^{2}+1\right)
\cdot p+(\alpha-1)^{2}}{\alpha-1}%
\]
is true, because it is equivalent to $0\leq\left(  p+2\alpha+p\cdot
\alpha-\alpha^{2}-1\right)  ^{2}$. Moreover, if both $\alpha\geq2$ and
$\frac{\alpha\left(  \alpha-1\right)  }{\alpha+1}\leq p$, then
\[
p+2\alpha+p\alpha-\alpha^{2}-1\geq\frac{\alpha\left(  \alpha-1\right)
}{\alpha+1}+2\alpha+\frac{\alpha^{2}\left(  \alpha-1\right)  }{\alpha
+1}-\alpha^{2}-1=\alpha-1>0\text{.}%
\]
In other words, the interval $\left[  \left(  p+1\right)  \cdot\alpha+\left(
\alpha-1\right)  ,\frac{\allowbreak\left(  \alpha^{2}+1\right)  \cdot
p+\left(  \alpha-1\right)  ^{2}}{\alpha-1}\right]  $ is not included in the
interval $\left[  (p+1)\cdot\alpha,p\cdot\alpha+2\sqrt{p\cdot\alpha+p}\right]
$.
\end{proof}

\begin{corollary}
Let $G$ be a connected $\mathbf{W}_{p}$ graph with $p\ne2$, $n=n(G)$ and
$\alpha=\alpha(G)$. If $\alpha-1\leq p$ and
\[
(p+1)\cdot\alpha\leq n\leq\frac{\allowbreak\left(  \alpha^{2}+1\right)  \cdot
p+\left(  \alpha-1\right)  ^{2}}{\alpha-1},
\]
then $I(G;x)$ is log-concave.
\end{corollary}

The graph $C_{5}$, which is a $\mathbf{W}_{2}$ graph, illustrates that the
independence polynomial $I(G;x)$ can be log-concave even when $(p+1)\cdot
\alpha(G)>n(G)$. Moreover, note that $I(G;x)$ may remain log-concave even when
$n(G)>p\cdot\alpha(G)+2\sqrt{p\cdot\alpha+p}$. For instance, consider the
well-covered graphs $G_{q}$ with $q\geq5$ from Figure \ref{fig2}. It is clear
that $\alpha\left(  G_{q}\right)  =3$, $m\left(  G_{q}\right)  =6+\frac
{q(q-1)}{2}$, and
\[
n\left(  G_{q}\right)  =q+4>p\cdot\alpha\left(  G_{q}\right)  +2\sqrt
{p\cdot\alpha(G_{q})+p}=7,
\]
and its independence polynomial
\[
I(G_{q};x)=1+\left(  q+4\right)  x+4qx^{2}+\left(  2q-2\right)  x^{3}%
\]
is log-concave. Notice that $G_{q}$ is $\frac{3}{2}$-quasi-regularizable and
belongs to $\mathbf{W}_{1}$.

\begin{figure}[h]
\setlength{\unitlength}{1cm} \begin{picture}(5,2)\thicklines
		\multiput(5,0.5)(1,0){3}{\circle*{0.29}}
		\multiput(5,1.5)(1,0){3}{\circle*{0.29}}
		\put(5,0.5){\line(1,0){2}}
		\put(5,0.5){\line(0,1){1}}
		\put(5,1.5){\line(1,0){2}}
		
		\put(6,0.5){\line(0,1){1}}

		\put(7,0.5){\line(0,1){1}}

		\put(5,0.1){\makebox(0,0){$x_{1}$}}
		\put(6,0.1){\makebox(0,0){$x_{3}$}}
		
		\put(5,1.85){\makebox(0,0){$x_{2}$}}
		\put(6,1.85){\makebox(0,0){$x_{4}$}}
		
		\put(7,0.1){\makebox(0,0){$y_{1}$}}
		\put(7,1.85){\makebox(0,0){$y_{2}$}}
		
		\multiput(7,0.5)(0.2,0){15}{\circle*{0.1}}
		\multiput(7,1.5)(0.2,0){15}{\circle*{0.1}}
		
		\multiput(10,0.5)(0,0.2){6}{\circle*{0.1}}
		
		\put(8.5,1){\makebox(0,0){$K_{q}$}}
		
		
	\end{picture}
\caption{$G_{q}$ is a well-covered $\frac{3}{2}$-quasi-regularizable graph.}%
\label{fig2}%
\end{figure}
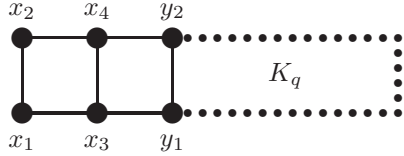

\begin{theorem}
\label{theorem17}Let $G$ be a $\mathbf{W}_{p}$ graph with $p\ne2$. If the
inequalities
\[
(p+1)\cdot\alpha\left(  H\right)  \leq n\left(  H\right)  \leq p\cdot
\alpha(H)+2\sqrt{p\cdot\alpha(H)+p},
\]
or%
\[
p\cdot\alpha(H)+2\sqrt{p\cdot\alpha(H)+p}<n\left(  H\right)  \leq
\frac{\allowbreak\left(  \alpha\left(  H\right)  ^{2}+1\right)  \cdot
p+\left(  \alpha\left(  H\right)  -1\right)  ^{2}}{\alpha\left(  H\right)  -1}%
\]
hold for every connected component $H$ of $G$, then $I(G;x)$ is log-concave.
\end{theorem}

\begin{proof}
If $G$ is connected, the theorem follows from Theorem \ref{log2}.

If $G$ is disconnected with $c(G)=q\geq2$, then $G$ is a disjoint union of
connected components $H_{i}$ for all $1\leq i\leq q$. By Theorem
\ref{disconnected}, $H_{i}\in\mathbf{W}_{p}$. By Theorem \ref{log2},
$I(H_{i};x) $ is log-concave. By Lemma \ref{log-uni}, we finally obtain that
\[
I(G;x)=I(H_{1};x)\cdot I(H_{2};x) \cdots I(H_{q};x)
\]
is log-concave as well.
\end{proof}

\begin{corollary}
\label{cor7} Let $G$ be a $\mathbf{W}_{p}$ graph with $p\ne2$. If
$(p+1)\cdot\alpha(G) =n(G)$ and $p\ge\frac{\alpha(G)^{2}}{4(\alpha(G)+1)}$,
then $I(G;x)$ is log-concave.
\end{corollary}

\begin{proof}
Clearly, if $(p+1)\cdot\alpha\left(  G\right)  =n\left(  G\right)  $, then the
same is true for every connected component $H$ of $G$. Hence, $(p+1)\cdot
\alpha\left(  H\right)  =n\left(  H\right)  \leq p\cdot\alpha\left(  H\right)
+ 2\sqrt{p\cdot\alpha(H)+p}$, whenever $p\geq\frac{\alpha(H)^{2}}%
{4(\alpha(H)+1)}$. In addition, Theorem \ref{disconnected} claims that every
$H\in\mathbf{W}_{p}$. Therefore, by Theorem \ref{log2}, $I(H;x)$ is
log-concave. Finally, by Lemma \ref{log-uni}, $I(G;x)=%
{\displaystyle\prod\limits_{H}}
I(H;x) $ is log-concave.
\end{proof}

The disjoint union of two graphs $G_{1}$ and $G_{2}$ is the graph $G=G_{1}\cup
G_{2}$ whose vertex set is the disjoint union of $V(G_{1})$ and $V(G_{2})$,
and whose edge set is the disjoint union of $E(G_{1})$ and $E(G_{2})$.
Specifically, $qG$ denotes the disjoint union of $q>1$ copies of the graph
$G$. The join (or Zykov sum) of $G_{1}$ and $G_{2}$ is the graph $G_{1}+G_{2}%
$, with vertex set $V(G_{1})\cup V(G_{2})$ and edge set $E(G_{1})\cup
E(G_{2})\cup\{v_{1}v_{2}:v_{1}\in V(G_{1}),v_{2}\in V(G_{2})\}$. The
lexicographic product $G[H]$ of the graphs $G$ and $H$ is defined as follows:
$V\left(  G[H]\right)  =V\left(  G\right)  \times V\left(  H\right)  $ and two
verices $(g_{1},h_{1})$ and $(g_{2},h_{2})$ are adjacent in $G[H]$ if and only
if either $g_{1}g_{2}\in E\left(  G\right)  $, or $g_{1}=g_{2}$ and
$h_{1}h_{2}\in E\left(  H\right)  $.

\begin{example}
For $m\geq2$ and $24\leq n\leq2452$, let $G=K_{m}$ and $H=4K_{10}+K_{n(4)}$,
where $K_{n(4)}$ is the complete $n$-partite graph where each of the $n$ parts
has $4$ vertices. Then $G$ is a $\mathbf{W}_{2}$ graph, while $H$ is a
$\mathbf{W}_{1}$ graph \cite[p. 240]{LM17}. By \cite[Theorem 8.2]{Plum}, the
lexicographic product $G[H]$ of $G$ and $H$ is a $\mathbf{W}_{2}$ graph. The
independence polynomials of $G$ and $H$ are as follows:
\begin{align*}
I(G;x)  & =1+mx,\\
I(H;x)  & =1+(40+4n)x+(600+6n)x^{2}+(4000+4n)x^{3}+(10000+n)x^{4}.
\end{align*}
By applying the independence polynomial formula for $G[H]$ as presented
in~\cite{BHN04}, we obtain:
\begin{align*}
I(G[H];x)  & =I(G;I(H;x)-1)=1+m\cdot(I(H;x)-1)\\
& =1+m(40+4n)x+m(600+6n)x^{2}+m(4000+4n)x^{3}+m(10000+n)x^{4}.
\end{align*}
Notice that this polynomial is not log-concave.
\end{example}

A well-covered graph $G$ having no isolated vertices and satisfying
$2\cdot\alpha(G)=n(G)$ is called \textit{very well-covered} \cite{Favaron1982}%
. Previously, it was established that for any integer $\alpha\geq8$, there
exist connected well-covered graphs $G$ with $\alpha\left(  G\right)  =\alpha
$, whose independence polynomials are not unimodal (hence, not log-concave)
\cite{LevMan2006}. It was also shown that the independence polynomial of a
very well-covered graph $G$ is unimodal when $\alpha\left(  G\right)  \leq9$
and is log-concave when $\alpha\left(  G\right)  \leq5$ \cite{LM18}. Now,
Corollary \ref{cor7} corroborates \cite[Theorem 2.4(v)]{LM18} and
\cite[Conjecture 3.10]{LM17} for the case $\alpha(G)\leq5$, as well.

\begin{corollary}
\label{cor11} Let $G$ be a very well-covered graph. If $\alpha(G)\le5$, then
the independence polynomial of $G$ is log-concave.
\end{corollary}

\begin{proof}
Let $\alpha=\alpha(G)$ and $I(G;x)=\sum_{k=0}^{\alpha}s_{k}x^{k}$ denote the
independence number and independence polynomial of $G$, respectively. Recall
that a very well-covered graph is a well-covered graph with order $2\alpha$.
To prove that $I(G;x)$ is log-concave, based on the proof of Theorem
\ref{log2}, it suffices to verify that the inequality
\[
k^{2}-\alpha\cdot k+\alpha+1\geq0
\]
holds for all $1\leq k\leq\alpha-1$. This condition is equivalent to requiring
$\alpha\leq5$. Therefore, the independence polynomial of $G$ is log-concave
for all $\alpha\leq5$.
\end{proof}

\begin{lemma}
\label{Lem1}(\cite[Corollary 2.3]{DLMP} and \cite[Theorem 4]{Staples}) Let $H
$ be a graph. Then $H\circ K_{p}$ is a $\mathbf{W}_{p}$ graph, but it is not a
$\mathbf{W}_{p+1}$ graph.
\end{lemma}

Consequently, this leads to the following.

\begin{corollary}
\label{corona_log} Let $H$ be a graph of order $n$. The polynomial $I(H\circ
K_{p};x)$ is log-concave for every $p\geq\frac{n^{2}}{4(n+1)}$.
\end{corollary}

\begin{proof}
Taking into account that $n(H\circ K_{p})=(p+1)\cdot n\left(  H\right)  $ and
$\alpha(H\circ K_{p})=n\left(  H\right)  $, we obtain that%
\[
n(H\circ K_{p})=(p+1)\cdot\alpha(H\circ K_{p}).
\]
Consequently, $I(H\circ K_{p};x)$ is log-concave for all $p\geq1$, in
accordance with Lemma \ref{Lem1} and Corollary \ref{cor7}.
\end{proof}

\begin{corollary}
\cite{DLMP} Let $G\circ\mathcal{H}$ be a clique corona graph, where
$\mathcal{H}=\{K_{p\left(  v\right)  }:v\in V(G) \text{ and } p\left(
v\right)  \geq1\}$. Let $p=\min\limits_{v\in V(G)}p\left(  v\right)  $. Then
$G\circ\mathcal{H}\in\mathbf{W}_{p}$.
\end{corollary}

\begin{corollary}
\label{GHpmin} Let $p=\min\limits_{v\in V(G)}p\left(  v\right)  $ and
$\mathcal{H}=\{K_{p\left(  v\right)  }:v\in V(G) \text{ and } p\left(
v\right)  \geq1\}$. The polynomial $I(G\circ\mathcal{H};x)$ is log-concave for
all $p\ge\frac{n(G)^{2}}{4(n(G)+1)}$.
\end{corollary}

As a consequence, the following corollary partially confirms that the
conjecture stated in \cite[Conjecture 4.2]{LM18} also holds true.

\begin{corollary}
\label{cor_tree} If $G$ is a well-covered tree with at least two vertices with
$\alpha(G)\le5$, then $I(G;x)$ is log-concave.
\end{corollary}

It is established that taking the corona of any graph $G$ with $K_{1}$ yields
the very well-covered graph $G\circ K_{1}$ (see \cite[Corollary 3]{TV92}). By
applying this method, we can generate an infinite family of very well-covered
trees based on any given tree. This brings us to the following.

\begin{corollary}
For any tree $T$ with $\alpha(T)\le5$, the independence polynomials of the
following graphs%
\[
T\circ K_{1},\left(  T\circ K_{1}\right)  \circ K_{1},\left(  \left(  T\circ
K_{1}\right)  \circ K_{1}\right)  \circ K_{1},\left(  \left(  \left(  T\circ
K_{1}\right)  \circ K_{1}\right)  \circ K_{1}\right)  \circ K_{1},\ldots
\]
are log-concave.
\end{corollary}

\section{Conclusion}

\label{sec4} This paper primarily addresses problems related to log-concavity
of $\mathbf{W}_{p}$ graphs. Our results indicate that the independence
polynomial $I(G; x)$ is log-concave whenever $p$ is sufficiently large
relative to the independence number $\alpha(G)$ of $G$. Furthermore, Zhu, in
\cite[Corollary 3.3]{Zhu}, demonstrated the log-concave preservation of the
independence polynomial of $G\circ K_{p}$ whenever the independence polynomial
of $G$ is log-concave. Theorem \ref{theorem17} leads us to the following.

\begin{prob}
What conditions on the $\mathbf{W}_{p}$ graph $G$ guarantee that the
independence polynomial $I(G;x)$ is log-concave or at least unimodal?
\end{prob}

In addition, Theorem \ref{mthm} motivates the following.

\begin{conjecture}
Let $G$ be a connected $\mathbf{W}_{2}$ graph. Then $G$ is $2$%
-quasi-regularizable if and only if $n(G)\geq3\cdot\alpha(G)$.
\end{conjecture}

\section*{Acknowledgment}

Parts of this work were carried out during a stay of the first and fourth
authors at the Vietnam Institute for Advanced Study in Mathematics (VIASM).
They would like to thank VIASM for its hospitality and generous support.
Additionally, their research is also partially supported by NAFOSTED (Vietnam)
under the grant number 101.04-2024.07.

\section*{Declarations}

\noindent\textbf{Conflict of interest/Competing interests} \newline The
authors declare that they have no competing interests\newline\noindent
\textbf{Ethical approval and consent to participate}\newline\noindent Not
applicable. \newline\noindent\textbf{Consent for publication}\newline\noindent
Not applicable. \newline\noindent\textbf{Availability of data, code and
materials}\newline\noindent Data sharing not applicable to this work as no
data sets were generated or analyzed during the current study. \newline%
\noindent\textbf{Authors' contribution}\newline All authors have contributed
equally to this work.\newline


\begin{thebibliography}{99}                                                                                               %
\bibitem {AMSE}Y. Alavi, P. J. Malde, A. J. Schwenk, P. Erd\"{o}s, \emph{The
vertex independence sequence of a graph is not constrained}, Congressus
Numerantium \textbf{58} (1987), 15--23.

\bibitem {Radcliffe}T. Ball, D. Galvin, K. Weingartner, \emph{Independent set
and matching permutations}, Journal of Graph Theory \textbf{99} (2022) 40--57.

\bibitem {Berge}C. Berge, \emph{Some common properties for regularizable
graphs, edge-critical graphs and B-graphs}, Annals of Discrete Mathematics
\textbf{12} (1982), 31--44.

\bibitem {BHN04}J. I. Brown, C. A. Hickman, R. J. Nowakowski, \emph{On the
location of roots of independence polynomials}, Journal of Algebraic
Combinatorics \textbf{19} (2004) 273--282.

\bibitem {CP}S. R. Campbell, M. D. Plummer, \emph{On well-covered 3
polytopes}, Ars Combinatoria \textbf{25} (1988) 215--242.

\bibitem {ChenWnag2010}S.-Y. Chen, H.-J. Wang, \emph{Unimodality of very
well-covered graphs}, Ars Combinatoria \textbf{97A} (2010) 509--529.

\bibitem {CS}M. Chudnovsky, P. Seymour, \emph{The roots of the independence
polynomial of a claw-free graph}, Journal of Combinatorial Theory, Series B
\textbf{97} (2007) 350--357.

\bibitem {Favaron1982}O. Favaron, \emph{Very well-covered graphs}, Discrete
Mathematics \textbf{42} (1982) 177--187.

\bibitem {FHN93}A. Finbow, B. Hartnell, R. Nowakowski, \emph{A
characterization of well-covered graphs of girth 5 or greater}, Journal of
Combinatorial Theory, Series B \textbf{57} (1993) 44--68.

\bibitem {FruchtHarary}R. Frucht, F. Harary, \emph{On the corona of two
graphs}, Aequationes Mathematicae \textbf{4} (1970) 322--324.

\bibitem {GU2}I. Gutman, F. Harary, \emph{Generalizations of the matching
polynomial}, Utilitas Mathematica \textbf{24} (1983) 97--106.

\bibitem {HL72}O. J. Heilmann, E. H. Lieb, \emph{Theory of monomer-dimer
systems} Communications in Mathematical Physics \textbf{25} (1972) 190--232.

\bibitem {DLMP}D. T. Hoang, V. E. Levit, E. Mandrescu, M. H. Pham, \emph{On
the unimodality of the independence polynomial of clique corona graphs}.
Available online at SSRN: http://dx.doi.org/10.2139/ssrn.4293649

\bibitem {DP}D. T. Hoang, M. H. Pham, \emph{The size of Betti tables of edge
ideal of clique corona graphs}, Archiv der Mathematik \textbf{118} (2022) 577--586.

\bibitem {HT}D. T. Hoang, T. N. Trung, \emph{A characterization of
triangle-free Gorenstein graphs and Cohen--Macaulayness of second powers of
edge ideals}, Journal of Algebraic Combinatorics \textbf{43} (2016) 325--338.

\bibitem {HL}C. Hoede, X. Li, \emph{Clique polynomials and independent set
polynomials of graphs}, Discrete Mathematics \textbf{125} (1994) 219--228.

\bibitem {Huh12}J. Huh, \emph{Milnor numbers of projective hypersurfaces and
the chromatic polynomial of graphs}, Journal of the American Mathematical
Society \textbf{25} (2012) 907--927.

\bibitem {KadLev}O. Kadrawi, V. E. Levit, \emph{The independence polynomial of
trees is not always log-concave starting from order 26}, Ars Mathematica
Contemporanea (2025). Available online at: https://doi.org/10.26493/1855-3974.3207.2ad

\bibitem {KG}J. Keilson, H. Gerber, \emph{Some results for discrete
unimodality}, Journal of American Statistical Association \textbf{334} (1971) 386--389.

\bibitem {LevMan2006}V. E. Levit, E. Mandrescu, \emph{Independence polynomials
of well-covered graphs: Generic counterexamples for the unimodality
conjecture}, European Journal of Combinatorics \textbf{27} (2006) 931--939.

\bibitem {LM18}V. E. Levit, E. Mandrescu, \emph{Independence polynomials and
the unimodality conjecture for very well-covered, quasi-regularizable, and
perfect graphs}, Graph Theory in Paris: Proceedings of a conference in Memory
of Claude Berge, 2007, 243--254.

\bibitem {LM07}V. E. Levit, E. Mandrescu, \emph{A family of graphs whose
independence polynomials are both palindromic and unimodal}, Carpathian
Journal of Mathematics \textbf{23} (2007) 108--116.

\bibitem {LM03}V.~E.~Levit, E.~Mandrescu, \emph{A family of well-covered
graphs with unimodal independence polynomials}, Congressus Numerantium
\textbf{165} (2003) 195--207.

\bibitem {LM17}V. E. Levit, E. Mandrescu, \emph{The independence polynomial of
a graph - a survey}, Proceedings of the 1st International Conference on
Algebraic Informatics, Aristotle University of Thessaloniki, Greece, (2005) 233--254.

\bibitem {LM04}V. E. Levit, E. Mandrescu, \emph{Very well-covered graphs with
log-concave independence polynomials}, Carpathian Journal of Mathematics
\textbf{20} (2004) 73--80.

\bibitem {LM2017}V. E. Levit, E. Mandrescu, \emph{The Roller-Coaster
conjecture revisited}, Graphs and Combinatorics \textbf{33} (2017) 1499--1508.

\bibitem {W2LM}V. E. Levit, E. Mandrescu, $1$\emph{-well-covered graphs
revisited}, European Journal of Combinatorics \textbf{80} (2019) 261--272.

\bibitem {MT03}T. S. Michael, W. N. Traves, \emph{Independence sequences of
well-covered graphs: non-unimodality and the Roller-Coaster conjecture},
Graphs and Combinatorics \textbf{19} (2003) 403--411.

\bibitem {Pinter}M. R. Pinter,\textit{\ }\emph{A class of planar well-covered
graphs with girth four}, Journal of Graph Theory \textbf{19} (1995) 69--81.

\bibitem {Pinter2}M. R. Pinter, \emph{Planar regular one-well-covered graphs},
Congressus Numerantium \textbf{91} (1992) 159--159.

\bibitem {Pinter1991}M. R. Pinter, $\mathbf{W}_{2}$\emph{\ graphs and strongly
well-covered graphs: two well-covered graph subclasses}, Vanderbilt Univ.
Dept. of Math. Ph.D. Thesis, 1991.

\bibitem {Plummer1970}M. D. Plummer, \emph{Some covering concepts in graphs},
Journal of Combinatorial Theory \textbf{8} (1970) 91--98.

\bibitem {Plum}M. D. Plummer, \emph{Well-covered graphs: survey}, Quaestiones
Mathematicae \textbf{16} (1993) 253--287.

\bibitem {S2}R. P. Stanley, \emph{Log-concave and unimodal sequences in
algebra, combinatorics, and geometry}, Annals of the New York Academy of
Sciences \textbf{576} (1989) 500--535.

\bibitem {StaplesThesis}J. W. Staples, \emph{On some subclasses of
well-covered graphs}, Ph.D. Thesis, 1975, Vanderbilt University.

\bibitem {Staples}J. W. Staples, \emph{On some subclasses of well-covered
graphs}, Journal of Graph Theory \textbf{3} (1979) 197--204.

\bibitem {TV92}J.~Topp, L.~Volkman, \emph{On the well--coveredness of products
of graphs}, Ars Combinatoria \textbf{33} (1992) 199--215.

\bibitem {Vi}R. Villarreal, Monomial Algebras, Monographs and Textbooks in
Pure and Applied Mathematics \textbf{238}, Marcel Dekker, Inc., New York, 2001.

\bibitem {Zhu}B. X. Zhu, \emph{Clique cover products and unimodality of
independence polynomials}, Discrete Applied Mathematics \textbf{206} (2016) 172--180.

\bibitem {Zhu1}B. X. Zhu, Y. Chen, \emph{Log-concavity of independence
polynomials of some kinds of trees}, Applied Mathematics and Computation
\textbf{342} (2019) 35--44.
\end{thebibliography}
\end{document}